\documentclass[11pt,twoside,reqno]{amsart}
\allowdisplaybreaks
\usepackage{amsmath,amstext,amssymb,epsfig}
\usepackage{graphicx}
\usepackage{mathrsfs}
\calclayout
\makeatletter
\renewcommand{\@seccntformat}[1]{\bf\csname the#1\endcsname.}
\renewcommand{\section}{\@startsection{section}{1}
	\z@{.7\linespacing\@plus\linespacing}{.5\linespacing}
	{\normalfont\upshape\bfseries\centering}}
\renewcommand{\@biblabel}[1]{\@ifnotempty{#1}{#1.}}
\makeatother
\theoremstyle{plain}
\newtheorem{thm}{Theorem}[section]

\newtheorem{prop}[thm]{Proposition}
\newtheorem{cor}[thm]{Corollary}

\theoremstyle{definition}
\newtheorem{ex}[thm]{Example}
\newtheorem{defn}[thm]{Definition}
\newtheorem{rem}{Remark}[section]

\usepackage{cancel}
\usepackage[parfill]{parskip}
\usepackage[german]{varioref}
\usepackage[all]{xy}
\usepackage{color}

\def \>{\succ}
\def \<{\prec}

\begin{document}	
	\title[Nil MANSURO\u{G}LU \textsuperscript{}, Bouzid Mosbahi \textsuperscript{}]{ On structures of BiHom-Superdialgebras and their derivations}
	\author{ Nil MANSURO\u{G}LU \textsuperscript{1},\quad Bouzid Mosbahi \textsuperscript{2}}
\address{\textsuperscript{1}Kırşehir Ahi Evran University, Faculty of Arts and Science,
Department of Mathematics, 40100 Kırşehir, Turkey}
   	\address{\textsuperscript{2}Department of Mathematics, Faculty of Sciences, University of Sfax, Sfax, Tunisia}
		
\email{\textsuperscript{1}nil.mansuroglu@ahievran.edu.tr}
\email{\textsuperscript{2}mosbahi.bouzid.etud@fss.usf.tn}
	
	
	
	\begin{abstract}
 BiHom-superdialgebras are clear generalization of Hom-superdialgebras. The purpose of this note is to describe and to survey structures of BiHom-superdialgebras. Then we derive derivations of BiHom-superdialgebras.

\end{abstract}
\maketitle \section{ Introduction}\label{introduction}
The associative superdialgebras are also known as diassociative superalgebras. In 2015, C. Wang, Q. Zhang and Z. Weiz described them  as a generalization of associative superalgebras. These authors   developed  many results on associative superdialgebra in the paper \cite{CZ}. As a generalization of associative superalgebras, they have two associative products and also they comply three other statements. On the other hand, many results of superdialgebras were investigated in \cite{LH4,LH5,LH6}. There are many application areas of these algebras, for instance, physics, classical geometry and  non-commutative geometry. For more details (see \cite{GraMakMenPan,Loday,Lod,MakZah,Poz,RikRakBas}). The principal beginning point of this paper is based on  the concept of BiHom-dialgebras introduced by Zahari and Bakayoko in \cite{BakZah}. In this paper, the authors studied BiHom-associative dialgebras and they gave the central extensions  and  the classifications of BiHom-associative dialgebras with dimension $n\geq 4$. By following their approach in \cite{BakZah}, we present the concepts on BiHom-superdialgebras. Moreover, we state some main properties on such algebras.

\section{ Structure of BiHom-superdialgebras}

By using some definitions on superdialgebras and Hom-dialgebras in \cite{LH6,Y}, we introduce some structure of BiHom-superdialgebras. Furthermore, we survey some important consequences on BiHom-superdialgebras which are similar to some statements on Hom-superdialgebras \cite{BakZah}.
\begin{defn}\label{def1}
Let $H$ be a superspace. If two even bilinear mapings $\dashv,\vdash: H \times H \longrightarrow H$  satisfy the next statements
\begin{align*}
&\text{(i)  }\quad p \vdash (q \dashv r) = (p \vdash q) \dashv r, \\
&\text{(ii)  }\quad p \dashv (q \dashv r) = (p \dashv q) \dashv r = p \dashv (q \vdash r), \\
&\text{(iii)}\quad p \vdash (q \vdash r) = (p \vdash q) \vdash r = (p \dashv q) \vdash r
\end{align*}
for every homogeneous elements $p, q, r \in H$, then  a 3-tuple $(H,\dashv,\vdash)$ is called a superdialgebra.
\end{defn}

\begin{defn}
Given a superspace $H$.  If two even bilinear mapings $\dashv,\vdash: H\times H\longrightarrow H$  and an even superspace homomorphism
$\alpha : H \longrightarrow H$ satisfy the next properties
\begin{align*}
&\text{(i)  }\quad\alpha(p \dashv q) = \alpha(p) \dashv \alpha(q),\quad \alpha(p \vdash q) = \alpha(p) \vdash \alpha(q),\\
&\text{(ii)  }\quad\alpha(p) \dashv(q \dashv r) = (p \dashv q) \dashv \alpha(r) =\alpha(p) \dashv (q \vdash r),\\
&\text{(iii)  }\quad\alpha(p) \vdash (q \vdash r) = (p \vdash q) \vdash \alpha(r) = (p \dashv q) \vdash \alpha(r),\\
&\text{(iv)  }\quad\alpha(p) \vdash (q \dashv r) = (p \vdash q) \dashv \alpha(r)
\end{align*}
for every homogeneous elements $p, q, r \in H,$ then  a tuple $(H,\dashv,\vdash\alpha)$ is said to be a Hom-superdialgebra.
\end{defn}
\begin{defn}
Given a superspace $H$.    If an even bilinear map $\cdot : H \times H \longrightarrow H$ and two even superspace homomorphisms $\alpha,\epsilon : H \longrightarrow H$ hold the next properties
\begin{align*}
&\text{(i)}\quad \alpha \cdot \epsilon = \epsilon \cdot \alpha,\\
&\text{(ii)}\quad \alpha(p\cdot q) = \alpha(p)\cdot \alpha(q),\: \epsilon(p\cdot q) = \epsilon(p)\cdot\epsilon(q),\\
&\text{(iii)}\quad \alpha(p)\cdot (q\cdot r) = (p\cdot q)\cdot\epsilon(r)
\end{align*}
for every homogeneous elements $p,q,r \in H,$ then a 4-tuple $(H, \cdot,\alpha,\epsilon)$ is said to be a BiHom-associative superalgebra.
\end{defn}
\begin{defn}\label{def2}
Given a superspace $H$.  If two even bilinear mapings $\dashv,\vdash: H \times H \longrightarrow H$ and two even superspace homomorphisms
$\alpha,\epsilon : H \longrightarrow H$ satisfy the following axioms
\begin{align*}
 &\text{(i)}\quad \alpha\circ \epsilon = \epsilon \circ\alpha, \\
&\text{(ii)}\quad\alpha(p \dashv q) = \alpha(p) \dashv \alpha(q),\quad \alpha(p \vdash q) = \alpha(p) \vdash\alpha(q),\\
&\text{(iii)}\quad\epsilon(p \dashv q) = \epsilon(p) \dashv \epsilon(q), \quad \epsilon(p \vdash q) = \epsilon(p) \vdash\epsilon(q),\\
&\text{(iv)}\quad(p \dashv q) \dashv \epsilon(r) = \alpha(p) \dashv (q \dashv r),\\
&\text{(v)}\quad(p \vdash q) \dashv \epsilon(r) = \alpha(p) \vdash (q \dashv r),\\
&\text{(vi)}\quad(p \dashv q) \vdash \epsilon(r) = \alpha(p) \dashv (q \vdash r),\\
&\text{(vii)}\quad(p \vdash q) \vdash \epsilon(r) = \alpha(p) \vdash (q \vdash r)
\end{align*}
for every homogeneous elements $p, q, r \in H$, then a 5-tuple $(H,\dashv,\vdash,\alpha,\epsilon)$ is called a BiHom-superdialgebra.
\end{defn}
Here $\alpha$ and $\epsilon$ are called structure maps of $H$. Since the maps $\alpha,\epsilon$ commute, for any integer $m,n$,  we denote by
$$\alpha^{m}\epsilon^{n} = \underbrace{\alpha \circ\ldots\circ \alpha}_{(m-\text{times})}\circ\underbrace{\epsilon \circ \ldots\circ \epsilon}_{(n-\text{times})}.$$
Particularly, $\alpha^{0}\epsilon^{0} = Id_H$, $\alpha^{1}\epsilon^{1}= \alpha \epsilon$ and  $\alpha^{-m}\epsilon^{-n}$ is the inverse of $\alpha^{m}\epsilon^{n}.$
\begin{defn}
Given  a BiHom-superdialgebra  $(H,\dashv,\vdash,\alpha,\epsilon)$, when  $\alpha$ and $\epsilon$ are bijection, then  $(H,\dashv,\vdash,\alpha,\epsilon)$ is called a regular BiHom-superdialgebra.
\end{defn}
\begin{rem}
 If  $p\vdash q = p\dashv q = p\cdot q$ for any $p,q\in H$, then each BiHom-associative superalgebra $H$ is a BiHom-superdialgebra.
\end{rem}
\begin{ex}
By taking $\epsilon = \alpha$, any Hom-superdialgebra becomes a BiHom-superdialgebra. By setting $\alpha = \epsilon = id$, any superdialgebra is a BiHom-superdialgebra.
\end{ex}

\begin{defn}
Given two BiHom-superdialgebras $(H_1,\dashv,\vdash,\alpha,\epsilon)$ and $(H_2, \dashv^{\prime}, \vdash^{\prime},\alpha^{\prime},\epsilon^{\prime})$ and let $g:H_1 \longrightarrow H_2$ be an even homomorphism.  $g$ is said to be a morphism of BiHom-superdialgebras if
$$g \circ\alpha = \alpha^{\prime}\circ g,\quad g\circ\epsilon =\epsilon^{\prime}\circ g$$ and $$g(p) \dashv^{\prime} g(q) = g(p\dashv q), \quad g(p) \vdash^{\prime} g(q) = g(p\vdash q)$$ for every $p,q \in H_1$.
\end{defn}
\begin{defn}
 Let $\alpha$ and $\epsilon$ be two morphisms. A BiHom-superdialgebra $(H,\dashv,\vdash,\alpha,\epsilon)$ is said to be a multiplicative BiHom-superdialgebra.
 \end{defn}
\begin{defn}
 if $\alpha$ and $\epsilon$ are bijective, we say that $(H,\dashv,\vdash,\alpha,\epsilon)$ is  a regular BiHom-superdialgebra.
\end{defn}

\begin{thm} \label{thm1}
Let $(H,\dashv,\vdash,\alpha,\epsilon)$ be a BiHom-superdialgebra and $\alpha^{\prime} ,\epsilon^{\prime}:H \longrightarrow H$ be two commuting endomorphisms. Consider $\triangleleft =\dashv\circ (\alpha^{\prime} \otimes \epsilon^{\prime})$ and $\triangleright =\vdash\circ (\alpha^{\prime} \otimes \epsilon^{\prime})$.  Then,
$H_{(\alpha^{\prime},\epsilon^{\prime})} = (H, \triangleleft, \triangleright, \alpha\alpha^{\prime}, \epsilon\epsilon^{\prime})$
is a BiHom-superdialgebra.
\end{thm}
\begin{proof}
Here, we will only prove one of conditions in Definition \ref{def2}. For each $p,q,r\in H$, we obtain
\begin{align*}
(p \triangleleft q) \triangleleft \epsilon\epsilon^{\prime}(r) - \alpha\alpha^{\prime}(p) \triangleleft (q \triangleright r) &= \alpha^{\prime}(\alpha^{\prime}(p) \dashv \epsilon^{\prime}(q)) \dashv \epsilon^{\prime}\epsilon \epsilon^{\prime}(r) - \alpha^{\prime}\alpha\alpha^{\prime}(p) \dashv \epsilon^{\prime}(\alpha^{\prime}(q) \vdash\epsilon^{\prime}(r))\\
&= (\alpha^{\prime}\alpha^{\prime}(p) \dashv \alpha^{\prime}\epsilon^{\prime}(q)) \dashv \epsilon(\epsilon^{\prime}\epsilon^{\prime})(r) - \alpha\alpha^{\prime}\alpha^{\prime}(p) \dashv (\alpha^{\prime}\epsilon^{\prime}(q) \vdash \epsilon^{\prime}\epsilon^{\prime}(r))
\end{align*}
 Since $H$ is a BiHom-superdialgebra, by Using Definition \ref{def2} (v), we obtain
 \begin{align*}
(p \triangleleft q) \triangleleft \epsilon\epsilon^{\prime}(r) - \alpha\alpha^{\prime}(p) \triangleleft (q \triangleright r) &= \alpha(\alpha^{\prime}\alpha^{\prime}(p))\dashv (\alpha^{\prime}\epsilon^{\prime}(q) \vdash \epsilon^{\prime}\epsilon^{\prime}(r) )- \alpha\alpha^{\prime}\alpha^{\prime}(p) \dashv (\alpha^{\prime}\epsilon^{\prime}(q) \vdash \epsilon^{\prime}\epsilon^{\prime}(r)).
\end{align*}
  The left  side is equal to zero, therefore,  the proof of the part (iv) of Definition \ref{def2} is completed. By using the similar argument to other axioms, we prove the theorem.
\end{proof}
Consequently, the theorem derives the following sequence of corollaries.
\begin{cor}
Given the multiplicative BiHom-superdialgebra $(H,\dashv,\vdash,\alpha,\epsilon)$. Then, $(H, \dashv \circ(\alpha^{n} \otimes \epsilon^{n}), \vdash \circ(\alpha^{n} \otimes \epsilon^{n}), \alpha^{n+1}, \epsilon^{n+1})$ is  a multiplicative BiHom-superdialgebra as well.
\end{cor}
\begin{proof}
It is clearly sufficient to take $\alpha^{\prime} = \alpha^{n}$ and $\epsilon^{\prime} = \epsilon^{n}$ in Theorem \ref{thm1}. Thus, this proves the corollary.
\end{proof}
\begin{cor}
Let $(H,\dashv,\vdash,\alpha)$  be a multiplicative Hom-superdialgebra and $\epsilon: H \longrightarrow H$ be an endomorphism of $H$.
Then, $(H, \dashv \circ(\alpha \otimes \epsilon), \vdash \circ(\alpha \otimes \epsilon), \alpha^{2},\epsilon)$
is also a BiHom-superdialgebra.
\end{cor}
\begin{proof}
By taking $\alpha^{\prime} = \alpha$, $\beta=Id_{H}$ and $\epsilon^{\prime}=\epsilon$ in Theorem \ref{thm1}, we prove the corollary.
\end{proof}
\begin{cor}
If $(H,\dashv,\vdash,\alpha,\epsilon)$  is a regular BiHom-superdialgebra, then $(H, \dashv \circ( \alpha^{-1} \otimes \epsilon^{-1} ), \vdash \circ(\alpha^{-1} \otimes \epsilon^{-1}))$ is a superdialgebra.
\end{cor}
\begin{proof}
The claim follows in the case where $\alpha^{\prime} = \alpha^{-1}$ and $\epsilon^{\prime} = \epsilon^{-1}$ of Theorem \ref{thm1}.
\end{proof}
\begin{cor}
Given a superdialgebra $(H,\dashv,\vdash)$  and two commuting homomorphisms $\alpha,\epsilon: H \longrightarrow H$. Then, $(H, \dashv \circ(\alpha \otimes \epsilon), \vdash \circ(\alpha \otimes \epsilon), \alpha, \epsilon)$ is a BiHom-superdialgebra.
\end{cor}
\begin{proof}
By taking $\alpha = \epsilon=Id_{H}$ and replacing $\alpha^{\prime}$ by $\alpha$ and $\epsilon^{\prime}$ by $\epsilon$ in Theorem \ref{thm1}, we obtain the result as required.
\end{proof}


\begin{defn}
Let $(H, \dashv, \vdash, \alpha, \epsilon)$ be a BiHom-superdialgebra and $S$ a subset of $H$.   If $S$ is stable under $\alpha$ and $\epsilon$ and $p \dashv q, p \vdash q \in S$ for every $p, q \in S$,  $S$ is called a
BiHom-supersubalgebra of $H$.
\end{defn}

\begin{defn}
Given a BiHom-superdialgebra $(H, \dashv, \vdash, \alpha, \epsilon)$ and  a BiHom-supersubalgebra $T$ of $H$. For every $p \in T, q \in H$,  we say $T$ is a left BiHom-ideal of $H$ if we have  $p \dashv q, p \vdash q\in T$. If $q \dashv p$ and $q \vdash p$ are in $T$, then $T$ is called a right BiHom-ideal of $H$ and $T$ is said to be
 a two sided BiHom-ideal of $H$ if $p \dashv q \;, \; q \dashv p, \;p \vdash q \;, \; q \vdash p  \in T$.
\end{defn}
\begin{ex}
Clearly, $T = \{0\}$ and $T = H$ are two-sided BiHom-ideals.
 If $\theta : H_{1} \longrightarrow H_{2}$ is a morphism of BiHom-superdialgebras, then the kernel $Ker\theta$ is a two sided BiHom-ideal in $H_{1}$ and the image $Im\theta$ is a BiHom-supersubalgebra of $H_{2}$.  Moreover,
if $T_{1}$ and $T_{2}$ are two sided BiHom-ideals of $H$, then  $T_{1} + T_{2}$ is two sided BiHom-ideal as well.
\end{ex}
\begin{prop}
Given a BiHom-superdialgebra $(H, \dashv, \vdash, \alpha, \epsilon)$  and  a two sided BiHom-ideal $T$ of   $(H, \dashv, \vdash, \alpha, \epsilon)$.
Then, $(H/T,  \overline{\dashv},  \overline{\vdash},   \overline{\alpha},   \overline{\epsilon})$ is a BiHom-superdialgebra where
$$  \overline{\alpha}( \overline{p}) =  \overline{\alpha(p)}, \quad  \overline{\epsilon}( \overline{p}) =  \overline{\epsilon(p)}  $$ and $$\overline{p} \; \overline{\dashv} \; \overline{q} =  \overline{p \dashv q},\quad \overline{p}\;  \overline{\vdash} \; \overline{q} =  \overline{p \vdash q}$$
for all $\overline{p}, \overline{q} \in H/T$.
\end{prop}
\begin{proof}
Here, We only prove right superassociativity. Similarly,  the others are  proved. For every $\overline{p}, \overline{q}, \overline{r} \in H/T$, we obtain
\begin{align*}
(\overline{p}\overline{\vdash} \overline{q})\overline{\vdash} \overline{\epsilon}(\overline{r}) -  \overline{\alpha}(\overline{p})\overline{\vdash}(\overline{q}\overline{\vdash} \overline{r}) &=\overline{(p \vdash q)} \overline{\vdash} \overline{\epsilon(r)} - \overline{ \alpha(p)} \overline{\vdash}\overline{ (q \vdash r)}\\
&= \overline{(p \vdash q) \vdash \epsilon(r) -  \alpha(p) \vdash (q \vdash r)}\quad \text{(by Definition \ref{def2} (vii))}\\
&= \overline{0}.
\end{align*}
Then, $(H/T, \overline{\dashv}, \overline{\vdash}, \overline{\alpha}, \overline{\epsilon})$ is BiHom-superdialgebra.
\end{proof}
\section{ Derivation of BiHom-superdialgebra}

Some significant concepts corresponding to derivations of BiHom-superdialgebras will be presented in this section.

\begin{defn}
Let $(H,\mu,\alpha,\epsilon)$ be a BiHom-superalgebra. A linear maping $d:H \longrightarrow H$ is said to be an
$\alpha^{m}\epsilon^{n}$-derivation of $(H,\mu,\alpha,\epsilon)$ if it holds
\begin{align*}
\alpha \circ d&= d \circ\alpha  \quad \text{and} \quad  d \circ \epsilon = \epsilon \circ d, \\
d \circ \mu(p, q) &= \mu(d(p), \alpha^{m}\epsilon^{n}(q)) +(-1)^{|p||d|} \mu(\alpha^{m}\epsilon^{n}(p), d(q))
\end{align*}
for every $p,q\in H$ and for each integers $m,n$.
\end{defn}
\begin{defn}
A linear maping $d:H \longrightarrow H$ on a BiHom-superalgebra is called a differential if\\
$$d(p.q) = dp.q +(-1)^{|p||d|} p.dq, \quad \text{for every} \quad p, q \in H,\quad d^{2} = 0$$ and
$$d \circ\alpha = \alpha \circ d \quad d \circ \epsilon= \epsilon \circ d.$$
\end{defn}
\begin{prop}
Given a differential BiHom-superalgebra $(H,.,\alpha,\epsilon,d)$  and the products $\dashv$ and $\vdash$ on $H$  by
\begin{align*}
p \dashv q= \alpha(p) . dq \quad and \quad p \vdash q = dp . \epsilon(q).
\end{align*}
Then $(H,\dashv,\vdash,\alpha,\epsilon)$ is a BiHom-superdialgebra where $\alpha$ and $\epsilon$ are idempotant structure maps.
\end{prop}
\begin{proof}
It is easy to show the BiHom-superassociativity. We have
\begin{align*}
\alpha(p) \dashv (q \dashv r) &= \alpha(p) \dashv (\alpha(q).d(r))\\
&= \alpha (\alpha(p)).d(\alpha(q).d(r))\\
&=\alpha (\alpha(p)).(d(\alpha(q)).d(r)+\alpha(q)d^2(r))\\
&=\alpha (\alpha(p)).(d(\alpha(q)).d(r))\\
&= (p \dashv q) \dashv \epsilon(r)
\end{align*}
for any $p,q,r \in H$. By using a similar way, we obtain
\begin{align*}
\alpha(p) \vdash (q \vdash r) &= \alpha(p) \vdash (d(q).\epsilon(r))\\
&=d(\alpha(p)).\epsilon(d(q).\epsilon(r))\\
&=\alpha(d(p)).(\epsilon(d(q)).\epsilon^2(r))\\
&=(d(p).\epsilon(d(q))).\epsilon^2(r)\\
&=(p\vdash q)\vdash \epsilon(r).
\end{align*}
\end{proof}
\begin{prop}
Let $(H,\dashv,\vdash,\alpha,\epsilon)$ be a BiHom-superdialgebra such that $\alpha(r) = \epsilon(r)$ for all  $r \in H$. For any $p,q \in H$, we describe the linear maping $ad^{(\alpha,\epsilon)}_{r} : H \longrightarrow H$ by
$$ad^{(\alpha,\epsilon)}_{r}(p) = p \dashv \epsilon(p) - (-1)^{|p||r|}\alpha(r) \vdash x.$$
Then $ad_{r}^{(\alpha,\epsilon)}$ is a derivation of $H$.
\end{prop}
\begin{proof}
We have
\begin{align*}
ad^{(\alpha,\epsilon)}_{r}(p) \dashv q + p \dashv ad^{(\alpha,\epsilon)}_{r}(q)
&= (p \dashv \epsilon(r) - (-1)^{|p||r|} \alpha(r) \vdash p) \dashv q + p \dashv (q \dashv \epsilon(r) - (-1)^{|q||r|} \alpha(r) \vdash q)\\
&= (p \dashv \epsilon(r)) \dashv q - (-1)^{|p||r|} (\alpha(r) \vdash p)\dashv q  \\
&+ p \dashv (q \dashv \epsilon(r)) - (-1)^{|q||p|} p \dashv (\alpha(r) \vdash q)\\
&= p \dashv (\epsilon(r) \dashv q )- (-1)^{|p||r|} \alpha(r) \vdash (p\dashv q)  \\
&+ (p \dashv q )\dashv \epsilon(r) - (-1)^{|q||r|} p \dashv (\alpha(r) \vdash q)\\
&= (p \dashv q) \dashv \epsilon(r)- (-1)^{|p||r|} \alpha(r) \vdash (p \dashv q)\\
&= ad^{(\alpha,\epsilon)}_{r} (p \dashv q).
\end{align*}
As a result, $ad_{r}^{(\alpha,\epsilon)}$  is a derivation.
\end{proof}

\begin{defn}
Given a BiHom-superdialgebra $(H,\dashv,\vdash,\alpha,\epsilon)$. A linear maping $d : H \longrightarrow H$ is called a $\alpha^{m}\epsilon^{n}$-derivation of $H$, if it holds
\begin{align*}
&\text{(i)}\quad d\circ\alpha = \alpha \circ d, \epsilon \circ d=\quad\quad d\circ\epsilon ,\\
&\text{(ii)}\quad d(p\dashv q) = \alpha^{m}\epsilon^{n}(p)\dashv d(q) + (-1)^{|p||d|}d(p) \dashv\alpha^{m}\epsilon^{n}(q),\\
&\text{(iii)}\quad d(p\vdash q) = \alpha^{m}\epsilon^{n}(p)\vdash d(q) + (-1)^{|p||d|}d(p) \vdash \alpha^{m}\epsilon^{n}(q)
\end{align*}
for each $ p,q \in H$ and for each integers $m,n$.
\end{defn}
The set of all $\alpha^{m}\epsilon^{n}$-derivation of $H$ is represented by $Der_{\alpha^{m}\epsilon^{n}}(H)$ and  the set of all derivation on BiHom-superdialgebra is represented by $$Der(H) = \bigoplus_{m,n\ge0}Der_{\alpha^{m}\epsilon^{n}}(H).$$\\
We define, for any $d,d^{\prime} \in Der(H)$, the even  bilinear maping $[,] : Der(H)\times Der(H) \longrightarrow Der(H)$ by $$[d,d^{\prime}]=d\circ d^{\prime}-(-1)^{|d||d^{\prime}|}d^{\prime}\circ d.$$
\begin{prop}\label{prop1}
For each $d\in (Der_{\alpha^{m}\epsilon^{n}}(H))_{i}$ and $d^{\prime} \in (Der_{\alpha^{m}\epsilon^{n}}(H))_{j}$, then $$[d,d^{\prime}] \in Der_{\alpha^{m+s}\epsilon^{n+t}} (H)_{|d|+|d^{\prime}|},$$
where $m+s,n+t \geq -1 \: \text{and} \:(i,j)\in \mathbb{Z}^{2}_{2}$.
\end{prop}
\begin{proof}
 For each $p,q \in H$, we have
\begin{align*}
[d,d^{\prime}](\mu(p,q))&= (d\circ d^{\prime}-(-1)^{|d||d^{\prime}|}d^{\prime}\circ d)(\mu(p,q))\\
&=d \circ d^{\prime}\mu(p,q) -(-1)^{|d||d^{\prime}|}d^{\prime}\circ d\mu(p,q)\\
&= d(\mu(d^{\prime}(p),\alpha^{s}\epsilon^{t}(q)) + (-1)^{|p||d^{\prime}|}\mu(\alpha^{s}\epsilon^{t}(p),d^{\prime}(q))
-(-1)^{|d||d^{\prime}|}d^{\prime}(\mu(d(p),\alpha^{k}\epsilon^{l}(q))\\
&+(-1)^{|d||p|}\mu(\alpha^{s}\epsilon^{t}(p),d(q)))\\
&=\mu(d\circ d^{\prime}(p),\alpha^{m}\epsilon^{n}\alpha^{s}\epsilon^{t}(q))
+ (-1)^{|d|(|d^{\prime}|+|p|)}\mu(\alpha^{m}\epsilon^{n}d^{\prime}(p),d(\alpha^{s}\epsilon^{t}(q)))\\
&+(-1)^{|p||d^{\prime}|}\mu(d(\alpha^{s}\epsilon^{t}(p),\alpha^{m}\epsilon^{n}d^{\prime}(q)))
+(-1)^{|p||d^{\prime}|+|p||d|}\mu(\alpha^{m}\epsilon^{n}\alpha^{s}\epsilon^{t}(p),d\circ d^{\prime}(q))) \\
&-(-1)^{|d||d^{\prime}|}\mu(d^{\prime}\circ d(p), \alpha^{s}\epsilon^{t}\alpha^{m}\epsilon^{n}(q)))
-(-1)^{|p||d^{\prime}|}\mu(\alpha^{s}\epsilon^{t}d(p),d^{\prime}\alpha^{m}\epsilon^{n}(q))\\
&-(-1)^{|d||d^{\prime}|+|p||d|}\mu(d^{\prime}(\alpha^{m}\epsilon^{n}(p),\alpha^{s}\epsilon^{t}d(q))
-(-1)^{|d||d^{\prime}|+|p||d|+|d||d^{\prime}|}\mu(\alpha^{s}\epsilon^{t}(p)\alpha^{m}\epsilon^{n}(p),d^{\prime}\circ d(q)).
\end{align*}
Since any two of maps $d, d^{\prime}, \alpha,\epsilon$ commute, we have
\begin{align*}
&d^{\prime}\circ\epsilon^{n}=\epsilon^{n}\circ d^{\prime},\quad
d^{\prime}\circ\alpha^{m}=\alpha^{m}\circ d^{\prime},\\
&d\circ\epsilon^{t}=\epsilon^{t}\circ d, \quad d\circ\alpha^{s}=\alpha^{s}\circ d.
\end{align*}
Therefore, we have
$$[d,d^{\prime}](\mu(p,q))=\mu([d,d^{\prime}](p),\alpha^{m+s}\epsilon^{n+t}(q))+(-1)^{|p|(|d^{\prime}|+|d|)}\mu(\alpha^{m+s}\beta^{n+t}(p),[d,d^{\prime}](q)).$$\\
In addition, it is easy to see that
\begin{align*}
[d,d^{\prime}]\circ \alpha&=(d\circ d^{\prime}-(-1)^{|d||d^{\prime}|}d^{\prime}\circ d)\circ \alpha\\
&=d\circ d^{\prime}\circ \alpha-(-1)^{|d||d^{\prime}|}d^{\prime}\circ d\circ \alpha\\
&=\alpha \circ d \circ d^{\prime}-(-1)^{|d||d^{\prime}|}\alpha \circ d^{\prime}\circ d\\
&= \alpha\circ[d,d^{\prime}]
\end{align*}
and
\begin{align*}
[d,d^{\prime}]\circ \epsilon&=(d\circ d^{\prime}-(-1)^{|d||d^{\prime}|}d^{\prime}\circ d)\circ\epsilon\\
&=d\circ d^{\prime}\circ \epsilon-(-1)^{|d||d^{\prime}|}d^{\prime}\circ d\circ \epsilon\\
&=\epsilon\circ d\circ d^{\prime}-(-1)^{|d||d^{\prime}|}\epsilon\circ d^{\prime}\circ d\\
&=\epsilon\circ[d,d^{\prime}]
\end{align*}
which occurs that
$$[d,d^{\prime}] \in (Der_{\alpha^{m+s},\epsilon^{n+t}}(D))_{(|d||d^{\prime}|)}$$ by taking $\mu=\dashv$ and $\mu=\vdash$ respectively.\\
Consequently, for each integers $k$ and $l$, represent by $$Der(H) = (\bigoplus_{m,n}Der_{\alpha^{m},\epsilon^{n}}(H))_{0}+(\bigoplus_{m,n}Der_{\alpha^{m},\epsilon^{n}}(H))_{1}.$$
\end{proof}

\begin{defn}
Let $(H,\dashv,\vdash,\alpha,\epsilon)$ be a BiHom-superdialgebra and $\gamma,\delta,\lambda$ elements in $\mathbb{C}$. A linear maping $d : H \longrightarrow H$ is a generalized $\alpha^{m}\epsilon^{n}$-derivation
or a $(\gamma,\delta,\lambda)-(\alpha^{m}\epsilon^{n})$-derivation of $H$
if for every $p,q \in H$, we have
\begin{align*}
&\text{(i)}\quad d\circ\alpha = \alpha \circ d, d\circ\epsilon =\epsilon \circ d,\\
&\text{(ii)}\quad\gamma(d(p\dashv q)) = \delta(d(p)\dashv\alpha^{m}\epsilon^{n}(q)) + (-1)^{|p||d|}\lambda(\alpha^{m}\epsilon^{n}(p)\dashv d(q)),\\
&\text{(iii)}\quad\gamma(d(p\vdash q)) = \delta(d(p)\vdash\alpha^{m}\epsilon^{n}(q)) + (-1)^{|p||d|}\lambda(\alpha^{m}\epsilon^{n}(p)\vdash d(q)).\\
\end{align*}
Denote by $Der^{\gamma,\delta,\lambda}(H)$ the set of $(\gamma,\delta,\lambda)-(\alpha^{m}\epsilon^{n})$-derivations of the  BiHom-superdialgebra $(H,\dashv,\vdash,\alpha,\epsilon)$.
\end{defn}
\begin{prop}
Let $d \in Der_{(\alpha^{m},\epsilon^{n})}^{(\gamma,\delta,\lambda)}(H) \quad$ and $\quad d^{\prime} \in Der_{(\alpha^{m^{\prime}},\epsilon^{n^{\prime}})}^{(\gamma^{\prime},\delta^{\prime},\lambda^{\prime})}(H)$. Then
$$[d, d^{\prime}] \in Der_{(\alpha^{m+m^{\prime}},\epsilon^{n+n^{\prime}})}^{(\gamma+\gamma^{\prime},\delta+\delta^{\prime},\lambda+\lambda^{\prime})}(H).$$
\end{prop}
\begin{proof}
 Applying the similar way to Proposition \ref{prop1} completes the proof.
\end{proof}
\begin{defn}
Let $(H,\dashv,\vdash,\alpha,\epsilon)$ be a BiHom-superdialgebra.
A linear maping $d : H \longrightarrow H$ is said to be an $\alpha^{k}\epsilon^{l}$-quasi derivation of the  BiHom-superdialgebra $(H,\dashv,\vdash,\alpha,\epsilon)$ if there is a derivation $d^{\prime} : H \longrightarrow H$
\begin{align*}
&\text{(i)}\quad d\circ\alpha = \alpha \circ d, \epsilon \circ d=\quad d\circ\epsilon,\\
&\text{(ii)} \quad d^{\prime}\circ\alpha = \alpha \circ d^{\prime}, \epsilon \circ d^{\prime}=\quad d^{\prime}\circ\epsilon \\
&\text{(iii)}\quad d^{\prime}(p\dashv q) = d(p)\dashv\alpha^{m}\epsilon^{n}(q) + (-1)^{|p||d|}(\alpha^{m}\epsilon^{n}(p)\dashv d(q)),\\
&\text{(iv)}\quad d^{\prime}(p\vdash q) = d(p)\vdash\alpha^{m}\epsilon^{n}(q) + (-1)^{|p||d|}(\alpha^{m}\epsilon^{n}(p)\vdash d(q))
\end{align*}
for all homogeneous elements $p,q \in H$.
\end{defn}
Denote by $QDer_{\alpha^{m}\epsilon^{n}}(H)$ the set of $\alpha^{m}\epsilon^{n}$-quasi derivations of $H$.
The set of all Quasi-derivations of $H$ is given by $$QDer(H) = (\bigoplus_{m,n\geq 0}QDer_{\alpha^{m},\epsilon^{n}}(H))_{\overline{0}}+(\bigoplus_{m,n\geq 0}QDer_{\alpha^{m},\epsilon^{n}}(H))_{\overline{1}}.$$

\end{document}